\documentclass{aims}
\usepackage{amsmath}
  \usepackage{paralist}
  \usepackage{graphics} 
  \usepackage{epsfig} 
\usepackage{graphicx}  \usepackage{epstopdf}
 \usepackage[colorlinks=true]{hyperref}
\hypersetup{urlcolor=blue, citecolor=red}

\def\currentvolume{X}
 \def\currentissue{X}
  \def\currentyear{200X}
   \def\currentmonth{XX}
    \def\ppages{X--XX}
     \def\DOI{10.3934/xx.xx.xx.xx}

\newtheorem{corollary}{Corollary}

\newtheorem{proposition}{Proposition}

\theoremstyle{definition}

\newtheorem{remark}{Remark}

\newtheorem{Satz}{Theorem} 
\newtheorem{Lemm}{Lemma} 
\newtheorem{Def}{Definition} 

\newtheorem{Ass}{Assumption}

\newtheorem  {examp}[Satz]     {Example}

\newcommand \N   {\mathbb{N}}
\newcommand \R   {\mathbb{R}}

\newcommand \K   {\mathcal{K}}
\newcommand \Kinf{\mathcal{K_\infty}}
\newcommand \PD  {\mathcal{P}}
\newcommand \KL  {\mathcal{KL}}
\newcommand \LL  {\mathcal{L}}

\newcommand \Iff   {\Leftrightarrow}

\newcommand{\lel}{\left\langle}
\newcommand{\rir}{\right\rangle}

\newcommand \eps {\varepsilon}

\title[\emph{i}ISS of bilinear systems] 
      {Characterizations of integral input-to-state stability for bilinear systems in infinite dimensions}

\author[Andrii Mironchenko and Hiroshi Ito]{}


 \email{andrii.mironchenko@uni-passau.de}
 \email{hiroshi@ces.kyutech.ac.jp}

\thanks{This work was supported by Japan Society for Promotion of Science (JSPS)}

\begin{document}
\maketitle

\centerline{\scshape Andrii Mironchenko }
\medskip
{\footnotesize
 \centerline{Department of Systems Design and Informatics}
   \centerline{Kyushu Institute of Technology}
   \centerline{680-4 Kawazu, Iizuka, Fukuoka 820-8502, Japan}
} 

\medskip

\centerline{\scshape Hiroshi Ito}
\medskip
{\footnotesize
 \centerline{Department of Systems Design and Informatics}
   \centerline{Kyushu Institute of Technology}
   \centerline{680-4 Kawazu, Iizuka, Fukuoka 820-8502, Japan}
}

\bigskip

 \centerline{(Communicated by the associate editor name)}

\begin{abstract}
For bilinear infinite-dimensional dynamical systems, 
we show the equivalence between uniform global asymptotic stability and integral input-to-state stability. We provide two proofs of this fact. One applies to general systems over Banach spaces. The other is restricted to Hilbert spaces, but is more constructive and results in an explicit form of iISS Lyapunov functions.  
\end{abstract}

\section{Introduction}

Stability and robustness are fundamental for control systems and typically they have been addressed within two different concepts. One is Lyapunov stability characterizing behavior of dynamical systems without inputs near equilibrium points. The other is input-output stability which neglects information about the state of a system and studies response of a system to external inputs. 
	In many cases it is not satisfactory to rely on one of them only. 
Input-to-state stability (ISS) \cite{Son89} unified these two concepts and provided powerful tools for stability analysis of ordinary differential equations (ODEs) such as ISS Lyapunov functions \cite{SoW95}, methods for design of nonlinear control systems \cite{KoA01,FrK08}, and the ISS small-gain theorem \cite{JTP94,JMW96} which has been extensively utilized both in theory and in applications to establish stability and robustness of networks of ISS systems \cite{DRW07,JiW08}.
%
%
The usefulness of ISS is widely recognized for ODE systems and led to generalizations, sophistications and abstractions to cover other types of control systems, such as time-delay systems, discrete-time and hybrid systems 
as well as trajectory-based systems (e.g. \cite{JiW01,PeJ06,InS02,KaJ11} to name a few).

The study of ISS of general infinite-dimensional systems and in particular 
of partial differential equations (PDEs) started relatively recently. 
In \cite{JLR08}, \cite{DaM13}, \cite{DaM13b}, \cite{Log13}, ISS of infinite-dimensional systems
\begin{equation}
\label{InfiniteDim}
\dot{x}(t)=Ax(t)+f(x(t),u(t)), \ x(t) \in X, u(t) \in U,
\end{equation}
has been addressed via methods of semigroup theory \cite{JaZ12}, \cite{CuZ95}. 
Here the state space $X$ and the space of input values $U$ are Banach spaces, $A:D(A) \to X$ is the generator of a $C_0$-semigroup over $X$ and $f:X \times U \to X$ is Lipschitz w.r.t. the first argument on any bounded set. Many classes of evolution PDEs, such as parabolic and hyperbolic PDEs are of this kind \cite{Hen81}, \cite{CaH98}.

In \cite{DaM13} sufficient conditions and a Lyapunov-based small-gain theorem for ISS of systems \eqref{InfiniteDim} have been developed, resulting in an efficient method to construct Lyapunov functions for interconnections of ISS systems and in this way to prove ISS of the interconnections.
The results from \cite{DaM13} have been transferred to not necessarily ISS impulsive systems and their interconnections in \cite{DaM13b}.

Frequency-domain methods have been applied to systems \eqref{InfiniteDim} with linear function $f$ in \cite{Log13} and to systems with sector-bounded nonlinearities in \cite{JLR08}.
A substantial effort has been devoted to constructions of ISS Lyapunov functions for nonlinear parabolic systems over $L_2$ spaces in \cite{MaP11}.
In \cite{PrM12} the construction of ISS-Lyapunov functions for time-variant linear systems of hyperbolic equations (balance laws) has been provided and these results have been applied to design a stabilizing boundary feedback
control for the Saint--Venant--Exner equations.
In \cite{DaM10} ISS of some classes of monotone parabolic systems has been considered.

%

In spite of powerful tools developed within ISS theory for ODEs, requirement of an ISS property is often too restrictive for practical systems, since in many cases boundedness of their trajectories is not guaranteed in the presence of inputs, i.e. their trajectories grow to infinity for inputs of large enough (but finite) magnitude. Such a situation is usual in biochemical processes, population dynamics, traffic flows etc. due to saturation and limitations in actuators and processing rates. Such systems are never ISS, but many of them enjoy a weaker robustness property, called integral input-to-state stability (iISS) \cite{Son98}. In \cite{ASW00} a Lyapunov type necessary and sufficient condition for an ODE system to be iISS has been proved. In \cite{Ito06,ItJ09,IJD13,AnA07,KaJ12} small-gain theorems for interconnections whose subsystems are not necessarily ISS have been developed. 


It is well-known that for linear ODE systems the notions of ISS and iISS coincide.
In \cite{Son98} it was proved by a direct construction of an iISS Lyapunov function, that bilinear ODE systems with Hurwitz autonomous matrices are always iISS, although many of them are not ISS. Bilinear systems have allowed us to understand a basic class of pure iISS systems and provided clues to dealing with more complicated iISS systems \cite{Son98,Ito10}.

In this work we are going to generalize the result from \cite{Son98} to bilinear infinite-dimensional systems \eqref{InfiniteDim} with bounded bilinear operators. On this way several difficulties arise. One of them is that the construction of Lyapunov functions, extending the original technique from \cite{Son98} directly works for systems with a Hilbert state space only. Therefore in order to prove equivalence between iISS and uniform global asymptotic stability for systems over Banach spaces, this paper develops a different method. Another difficulty is a need to use various density arguments, since the direct check of the properties of Lyapunov functions on the whole state space is often not possible. The work leads to an observation that non-uniformity over the spatial variables of a PDE system can make ISS fragile, while the system remains iISS.

The structure of the paper is as follows. 
Having introduced basic stability notions in Section~\ref{Problem_Formulation}, we define in Section~\ref{LyapFunk} iISS Lyapunov functions and 
prove that existence of an iISS Lyapunov function for a system \eqref{InfiniteDim} implies iISS of this system. ISS Lyapunov functions are also defined in a dissipative form so that they become a special case of iISS Lyapunov functions, while ISS Lyapunov functions are defined in an implicative form in \cite{DaM13}. In Section~\ref{LyapFunk} we show that both formulations of ISS Lyapunov functions coincide, provided that the operator $A$ generates an analytic semigroup and certain additional conditions on a nonlinearity hold. 
In Section~\ref{SectLinSys} we investigate iISS of bilinear systems. After a short discussion of infinite-dimensional linear systems, we prove that bilinear systems with bounded input operators which are uniformly globally asymptotically stable for a zero input are necessarily iISS. First we address this question for systems whose state space is an arbitrary Banach space. Next  we establish the abovementioned theorem for Hilbert spaces, which results in an explicit construction of an iISS Lyapunov function for bilinear systems. 
We illustrate our findings on an example of a parabolic system in Section~\ref{sec:Example} 
and conclude the paper in Section~\ref{sec:Conclusion}. In Appendix we prove two technical results.\footnote{The shortened preliminary version of this paper was published in
the 53rd IEEE Conference on Decision and Control \cite{MiI14}.}

We use the following notation throughout the paper. For linear normed spaces $X,Y$ let $L(X,Y)$ be the space of bounded linear operators from $X$ to $Y$ and $L(X):=L(X,X)$. A norm in these spaces we denote by $\| \cdot \|$. 
By $C_b(X,Y)$ we denote the space of bounded continuous functions from $X$ to $Y$, equipped with the standard $\sup$-norm, $C_b(X):=C(X,X)$.

We define $\R:=(-\infty,\infty)$ and $\R_+:=[0,\infty)$. Let $\N$ denote the set of natural numbers. Let $L_p(0,d)$, $p \geq 1$ be a space of $p$-th power integrable functions $f:(0,d) \to \R$ with the norm $\|f\|_{L_p(0,d)}=\left( \int_0^d{|f(x)|^p dx} \right)^{\frac{1}{p}}$.

\section{Problem formulation}
\label{Problem_Formulation}

Consider a system \eqref{InfiniteDim} and assume throughout the paper that $X$ and $U$ are Banach spaces and
$f(0,0)=0$, i.e., $x \equiv 0$ is an equilibrium point of the unforced system \eqref{InfiniteDim}. 
Let $\phi(t,\phi_0,u)$ denote the state of a system 
\eqref{InfiniteDim} at moment $t\in\R_+$ associated with 
an initial condition $\phi_0\in X$ at $t=0$, and input $u\in U_c$, 
where $U_c$ is a linear normed space of admissible inputs equipped 
with a norm $\|\cdot\|_{U_c}$. 


We use the following classes of comparison functions
\begin{equation*}
\begin{array}{ll}
{\PD} &:= \left\{\gamma:\R_+\rightarrow\R_+\left|\ \gamma\mbox{ is continuous, }  \gamma(0)=0 \mbox{ and } \gamma(r)>0 \mbox{ for } r>0 \right. \right\} \\
{\K} &:= \left\{\gamma \in \PD \left|\ \gamma \mbox{ is strictly increasing}  \right. \right\}\\
{\K_{\infty}}&:=\left\{\gamma\in\K\left|\ \gamma\mbox{ is unbounded}\right.\right\}\\
{\LL}&:=\{\gamma:\R_+\rightarrow\R_+\left|\ \gamma\mbox{ is continuous and strictly decreasing with } \lim\limits_{t\rightarrow\infty}\gamma(t)=0\right. \}\\
{\KL} &:= \left\{\beta:\R_+\times\R_+\rightarrow\R_+\left|\ \beta \mbox{ is continuous, } \beta(\cdot,t)\in{\K},\ \forall t \geq 0, \  \beta(r,\cdot)\in {\LL},\ \forall r > 0\right.\right\} \\
\end{array}
\end{equation*}
%

We consider mild solutions of \eqref{InfiniteDim}, i.e. solutions of the integral equation
\begin{equation}
\label{InfiniteDim_Integral_Form}
x(t)=T(t)x(0) + \int_0^t T(t-s)f(x(s),u(s))ds. 
\end{equation}
belonging to the class $C([0,\tau],X)$ for some $\tau>0$.
Here $\{T(t),\ t \geq 0\}$ is a $C_0$-semigroup on a Banach space $X$
with an infinitesimal generator $A:D(A) \to X$, $Ax =\lim\limits_{t \rightarrow +0}{\frac{1}{t}(T(t)x-x)}$.

\begin{Def}
We call $f:X \times U \to X$ Lipschitz continuous on bounded subsets of $X$, uniformly w.r.t. the second argument if $\forall w>0 \; \exists L(w)>0$, such that $\forall x,y: \|x\|_X \leq w,\ \|y\|_X \leq w$, $\forall v \in U$
\begin{eqnarray}
\|f(y,v)-f(x,v)\|_X \leq L(w) \|y-x\|_X.
\label{eq:Lipschitz}
\end{eqnarray}
\end{Def}

We will use the following assumption concerning nonlinearity $f$ throughout the paper
\begin{Ass}
\label{Assumption1}
We assume that $f:X \times U \to X$ is Lipschitz continuous on bounded subsets of $X$, uniformly w.r.t. the second argument and that $f(x,\cdot)$ is continuous for all $x \in X$.
\end{Ass}
Assumption~\ref{Assumption1} ensures that the mild solution of \eqref{InfiniteDim} exists and is unique, according to a variation of a classical existence and uniqueness theorem \cite[Proposition 4.3.3]{CaH98}.

Next we introduce stability properties for the system \eqref{InfiniteDim}.
\begin{Def}
System \eqref{InfiniteDim} is {\it globally asymptotically
stable at zero uniformly with respect to state} (0-UGASs), if $\exists \beta \in \KL$, such that $\forall \phi_0 \in X$, $\forall t\geq 0$ it holds
\begin{equation}
\label{UniStabAbschaetzung}
\left\| \phi(t,\phi_0,0) \right\|_{X} \leq  \beta(\left\| \phi_0 \right\|_{X},t) .
\end{equation}
%
\end{Def}
If undisturbed (for $u=0$) system \eqref{InfiniteDim} is merely locally stable and globally attractive, then \eqref{InfiniteDim} is called 0-GAS (which is a weaker notion than 0-UGASs).

%


Stability properties of \eqref{InfiniteDim} with respect to external inputs can be studied by means of the notion of input-to-state stability \cite{DaM13}:
\begin{Def}
\label{Def:ISS}
System \eqref{InfiniteDim} is called {\it input-to-state stable
(ISS) w.r.t. space of inputs $U_c$}, if there exist $\beta \in \KL$ and $\gamma \in \K$ 
such that the inequality
\begin {equation}
\label{iss_sum}
\begin {array} {lll}
\| \phi(t,\phi_0,u) \|_{X} \leq \beta(\| \phi_0 \|_{X},t) + \gamma( \|u\|_{U_c})
\end {array}
\end{equation}
holds $\forall \phi_0 \in X$, $\forall u\in U_c$ and $\forall t\geq 0$.

We emphasize that the above definition does not yet exactly 
correspond to ISS of finite dimensional systems \cite{Son08} since 
Definition \ref{Def:ISS} allows flexibility in the choice of the space $U_c$, which reflects a kind of dependence of inputs on time. 
\textit{System \eqref{InfiniteDim} is called ISS}, without expressing the 
normed space of inputs explicitly, if it is ISS w.r.t. $U_c=C_b(\R_+,U)$ 
endowed with a usual supremum norm. 
This terminology follows that of ISS for finite dimensional systems. 

\end{Def}


The central notion in this paper is
\begin{Def}
\label{Def:iISS}
System \eqref{InfiniteDim} is called integral input-to-state stable (iISS) if there exist $\theta \in \Kinf$, $\mu \in \K$ and $\beta \in \KL$ 
such that the inequality 
\begin{equation}
\label{iISS_Estimate}
\|\phi(t,\phi_0,u)\|_X \leq \beta(\|\phi_0\|_X,t) + 
\theta\left(\int_0^t \mu(\|u(s)\|_U)ds\right)
\end{equation}
holds $\forall \phi_0 \in X$, $\forall u\in U_c=C_b(\R_+,U)$ and $\forall t\geq 0$.
\end{Def}
Note that for iISS we did not want to introduce the freedom to choose spaces of input
functions which are not considered by this paper.

In the next section we will provide a Lyapunov sufficient condition for iISS of systems \eqref{InfiniteDim}.


\section{Lyapunov characterization of {\rm i}ISS}
\label{LyapFunk}

\begin{Def}\label{def:iISSV}
A continuous function $V:X \to \R_+$ is called an \textit{iISS Lyapunov function},  if there exist
$\psi_1,\psi_2 \in \Kinf$, $\alpha \in \PD$ and $\sigma \in \K$ 
such that 
\begin{equation}
\label{LyapFunk_1Eig}
\psi_1(\|x\|_X) \leq V(x) \leq \psi_2(\|x\|_X), \quad \forall x \in X
\end{equation}
and Lie derivative of $V$ along the trajectories of the system \eqref{InfiniteDim} satisfies 
\begin{equation}
\label{DissipationIneq}
\dot{V}_u(x) \leq -\alpha(\|x\|_X) + \sigma(\|u(0)\|_U)
\end{equation}
for all $x \in X$ and $u\in U_c$, 
where the Lie derivative of $V$ corresponding to the input $u$ is defined by
\begin{equation}
\label{LyapAbleitung}
\dot{V}_u(x)=\limsup \limits_{t \rightarrow +0} {\frac{1}{t}(V(\phi(t,x,u))-V(x)) }.
\end{equation}
Furthermore, if 
\begin{equation}
\label{eq:ISSalphsig}
\lim_{\tau\to\infty}\alpha(\tau)=\infty 
\,\mbox{ or }\, 
\liminf_{\tau\to\infty}\alpha(\tau)\geq
\lim_{\tau\to\infty}\sigma(\tau)
\end{equation}
holds, then $V$ is called an \textit{ISS Lyapunov function} for \eqref{InfiniteDim}. 
\end{Def}

\begin{remark}
For verification of stability properties like 0-UGASs, iISS and ISS, it suffices to assume $V$ continuous with respect to $x$ \cite[p. 84]{Hen81}. Thus, $t \mapsto V(x(t))$ is only continuous as well and the derivative of this function may not exist in the classical or 'almost everywhere' sense. However, the $\limsup$ in \eqref{LyapAbleitung} still exists, although it is possible that $\dot{V}_u(x) = -\infty$ or $\dot V_u(x)=\infty$. The latter case is excluded by \eqref{DissipationIneq}.
Still, for the ease of computation of $\dot{V}_u(x)$ it is common to pick $V$ locally Lipschitz in $x$. 
\end{remark}

\begin{remark}
In the Definition~\ref{def:iISSV} we give an estimate for a directional derivative of $V(\cdot)$ at any state $x$ and for any input $u$, which will be applied to this state. If $V$ is differentiable, it is possible to compute the derivatives along the trajectory not only at time $0$, but also at any time $t$:
\begin{eqnarray*}
\frac{d}{dt} V (\phi(t, \phi_0, u)) &=& \lim \limits_{\tau \rightarrow 0} {\frac{1}{\tau}\Big(V(\phi(t+\tau,\phi_0,u))-V(\phi(t,\phi_0,u))\Big) }\\
&=& \limsup\limits_{\tau \rightarrow +0} {\frac{1}{\tau}\Big(V(\phi(\tau,\phi(t,\phi_0,u),u_t))-V(\phi(t,\phi_0,u))\Big) }\\
&=& \dot{V}_{u_t}(\phi(t, \phi_0, u)).
\end{eqnarray*}
Here we have exploited the semigroup property $\phi(t+\tau,\phi_0,u)=\phi(\tau,\phi(t,\phi_0,u),u_t)$,
where $u_t(s) = u(t + s)$ for $s,t \geq 0$.
\end{remark}

Since Definition 5 assumes merely that $V$ is continuous, $V(\phi(t,\phi_0,u))$ is not guaranteed to be absolutely continuous in $t$. First, we shall confirm that the merely continuous function $V$ can give a traditional tool for estimating trajectories.  This will be done by extending of the comparison principles \cite[Lemma 4.4]{LSW96} and \cite[Corollary IV.3]{ASW00}. 

For this purpose, for any continuous function $y: \R\to\R$, let $\dot{y}$ denote the right upper Dini derivative of $y$, i.e. $\dot{y}(t):=\limsup_{h \to +0}\frac{y(t+h)-y(t)}{h}$.
The extension of \cite[Lemma 4.4]{LSW96} is as follows (the proof is postponed to Appendix~\ref{app:Comparison_Principle})
\begin{Lemm}
\label{thm:ComparisonPrinciple}
Let $y: \R_+ \to \R_+$ be a continuous function defined for all $t \geq 0$ and satisfying differential inequality
\begin{eqnarray}
\dot{y} \leq -\alpha(y(t)) \quad \forall t>0,
\label{eq:ComparisonPrinciple}
\end{eqnarray}
for some continuous and positive definite function $\alpha$.
Then there exists a $\beta \in \KL$ so that
\begin{eqnarray}
y(t) \leq \beta(y(0),t) \quad \forall t \geq 0.
\label{eq:ComparisonPrinciple_FinalEstimate}
\end{eqnarray}
\end{Lemm}

By virtue of this lemma we obtain an extension of \cite[Corollary IV.3]{ASW00}:
\begin{corollary}
\label{thm:ComparisonPrinciple-2}
Let $\tilde{t} \in (0,\infty]$ and let $y: [0,\tilde{t}) \to \R_+$ be a continuous function satisfying differential inequality
\begin{eqnarray}
\dot{y} \leq -\alpha(y(t)) + v(t) \quad \forall t \in (0,\tilde{t}),
\label{eq:ComparisonPrinciple-2}
\end{eqnarray}
for some continuous and positive definite function $\alpha$ and some measurable locally essentially bounded function 
$v:[0,\tilde{t}) \to \R_+$. 
Then for all $t \in [0,\tilde{t})$ it holds that
\begin{eqnarray}
y(t) \leq \beta(y(0),t) + \int_0^t 2 v(s)ds.
\label{eq:Comparison-2}
\end{eqnarray}
\end{corollary}
The proof of Corollary~\ref{thm:ComparisonPrinciple-2} goes along the lines of \cite[Corollary IV.3]{ASW00} by using Lemma~\ref{thm:ComparisonPrinciple} instead of \cite[Lemma 4.4]{LSW96}. This result helps to prove the following:

\begin{proposition}\label{PropSufiISS}
If there exist an iISS (resp. ISS) Lyapunov function for \eqref{InfiniteDim}, then \eqref{InfiniteDim} is iISS (resp. ISS).
%
\end{proposition}

\begin{proof}
Let $\tau \in (0,\infty]$ be such that $[0,\tau)$ is the maximal time interval 
over which a system \eqref{InfiniteDim} admits a solution. 
For a given initial condition $\phi_0\in X$ and 
a given input $u\in U_c$, let $y(t)=V(x(t))$, where $x(t)=\phi(t,\phi_0,u)$. 
With the help of \eqref{LyapFunk_1Eig}, 
definition \eqref{LyapAbleitung} and property 
\eqref{DissipationIneq} imply 
\begin{align*}
\dot{y}(t) \leq -\alpha(\psi_2^{-1}(y(t)))
+ \sigma(\|u(t)\|_U)  
\end{align*}
for each $t\in [0,\tau)$.  
From Corollary~\ref{thm:ComparisonPrinciple-2} it follows the existence of 
$\hat{\beta}\in\KL$ satisfying 
\begin{align*}
y(t) \leq \hat{\beta}(y(0),t)+\int_0^t
2\sigma(\|u(s)\|_U)ds . 
\end{align*}
Again, using \eqref{LyapFunk_1Eig} we have 
\begin{align}
\label{iISSprop_in_y}
\psi_1(\|x(t)\|_X) \leq \tilde{\beta}(\|x(0)\|_X,t)
+\int_0^t 2\sigma(\|u(s)\|_U)ds ,   
\end{align}
where $\tilde{\beta}=\hat{\beta}(\psi_2(\cdot),\cdot)\in\KL$. 
Since $\psi_1$ is of class $\K_\infty$ and satisfies 
$\psi_1^{-1}(a+b)\leq \psi_1^{-1}(2a)+\psi_1^{-1}(2b)$ for 
any $a,b\in\R_+$, 
we arrive at \eqref{iISS_Estimate} 
for all $t\in[0,\tau]$ 
with $\theta(s)=\psi_1^{-1}(2s)$ and $\mu(s)=2\sigma(s)$, 
and $\beta(s,r)=\psi_1^{-1}(2\tilde{\beta}(s,r))$.

Assume that $\tau$ is finite. Then, since Assumption~\ref{Assumption1} holds, the variation of \cite[Theorem 4.3.4]{CaH98}
implies that the solution $x(\cdot)$ is unbounded near $t=\tau$, which is false due to \eqref{iISSprop_in_y}. 
%
Therefore, $\tau=+\infty$ and property \eqref{iISS_Estimate} holds for 
all $t\geq 0$, and thus \eqref{InfiniteDim} is iISS.

Finally, we can prove ISS when \eqref{eq:ISSalphsig} holds as in the finite-dimensional case \cite{SoW95}, with the help of 
\cite[Theorem 1 and Proposition 5]{DaM13}
and the technique used in the last paragraph of the proof of Theorem~\ref{thm:Dissipative_Implicative}.
 %
\end{proof}

In Definition \ref{def:iISSV} 
the notion of ISS Lyapunov function is introduced in a so-called dissipative form 
in parallel to the notion of iISS Lyapunov function. 
In the preceding study \cite{DaM13} on ISS, however, a so-called implicative 
form has been used:
\begin{Def}
A continuous function $V:X \to \R_+$ is called an 
\textit{ISS Lyapunov function in an implicative form}, if there exist
$\psi_1,\psi_2 \in \Kinf$, $\eta \in \Kinf$ and $\gamma \in \K$ 
such that \eqref{LyapFunk_1Eig} holds and system \eqref{InfiniteDim} satisfies 
\begin{equation}
\label{GainImplikation}
\|x\|_X \geq \gamma(\|u(0)\|_U) \Rightarrow   \dot{V}_u(x) \leq -\eta(\|x\|_X) 
\end{equation}
for all $x \in X$ and $u\in U_c$.
\end{Def}
For finite-dimensional systems existence of an ISS Lyapunov function in a dissipative form is equivalent to existence of an ISS Lyapunov function 
in an implicative form, see \cite[Remark 2.4, p. 353]{SoW95} and \cite{SoT95}. 
For infinite-dimensional systems, this statement holds under some additional assumptions on nonlinearity $f$.

\begin{Satz}
\label{thm:Dissipative_Implicative}
Let nonlinearity $f$ and Lyapunov function candidate $V$ satisfy the conditions
\begin{enumerate}
	\item $f$ be Lipschitz on bounded subsets of $X$ uniformly w.r.t. the second argument.
  \item $A$ generate an analytic semigroup.
	\item $V$ be Fr\'echet differentiable in $X$ and its derivative $\frac{\partial V}{\partial x}$ be bounded on bounded balls. 
	\item $f$ and $V$ admit the existence of $p \in \K$ and $q \in \Kinf$ satisfying
\begin{align}
\left\| \frac{\partial V}{\partial x}f(0,v)\right\|_X \hspace{-1.3ex}\leq p(\| x\|_X)+q(\| v\|_U) 
, \ \forall x \in X, \ v \in U . 
\label{eq:Boundedness_wrt_u}
\end{align}
	\end{enumerate}
If $V$ is an ISS Lyapunov function in an implicative form for \eqref{InfiniteDim}, then $V$ is also an ISS Lyapunov function for \eqref{InfiniteDim} in a dissipative form. And if $V$ is an ISS Lyapunov function for \eqref{InfiniteDim} in a dissipative form, then there exist an ISS Lyapunov function (possibly different from $V$) for \eqref{InfiniteDim} in an implicative form.

\end{Satz}
The proof is given in Appendix~\ref{app:Implication_Dissipation}. 

\begin{remark}
Importantly, for infinite-dimensional systems, posing 
\eqref{eq:Boundedness_wrt_u} is fundamentally less restrictive than 
$\|f(0,v)\|_X\leq q(\| v\|_U)$ since 
norms in infinite-dimensional spaces are in general not equivalent, as opposed to 
finite-dimensional systems.
\end{remark}

When considering the approximations of dynamical systems, the following property is useful
\begin{Def}
We say that $\Sigma$ depends continuously on inputs and on initial states, if
$\forall x \in X, \forall u \in U_c, \forall \tau>0$ and $\forall \eps>0$ there exist $\delta=\delta(x,u,\tau,\eps)>0$, such that $\forall x' \in X: \|x-x'\|_X< \delta$ and 
$\forall u' \in U_c: \|u-u'\|_{U_c}< \delta$ it holds
\[
\|\phi(t,x,u)-\phi(t,x',u')\|_X< \eps \quad \forall t \in [0,\tau].
\]
If under the same assumptions 
\[
\|\phi(t,x,u)-\phi(t,x',u)\|_X< \eps \quad \forall t \in [0,\tau]
\]
holds, $\Sigma$ is called continuously dependent on initial states.
\end{Def}

In many cases it is hard to compute the derivative of a Lyapunov function for all $x \in X$  directly, but it is much more convenient to differentiate $V$ on some dense subspaces of $X$ and $U_c$, 
and to verify dissipation inequality \eqref{DissipationIneq} on the 
dense subspaces. This directly leads to iISS/ISS of a system over these dense subspaces provided they are invariant w.r.t. the flow of the system. We will show next, that under a
natural assumption, this already implies iISS of the system on 
the original state and input spaces. 
For this purpose, 
let $\Sigma:=(X,U_c,\phi)$ denote system \eqref{InfiniteDim} defined with 
the transition map $\phi$ corresponding to the spaces $X$, $U_c$. 
Let $\hat{\Sigma}:=(\hat{X},\hat{U}_c,\phi)$ be a restriction of a system \eqref{InfiniteDim} 
to the state space $\hat{X}$ and the input space $\hat{U}_c$
which are dense linear normed subspaces of $X$ and $U_c$, endowed with the norms in original spaces $\|\cdot\|_X$ and 
$\|\cdot\|_{U_c}$, respectively. 
In particular, this assumes that $\hat{X}$ is invariant under the flow $\phi$ for all inputs from 
$\hat{U}_c$.

Next we state two propositions, which will be used in the sequel. 

\begin{proposition}
\label{DensityArg}
Let $\Sigma$ depend continuously on inputs and on initial states and let $\hat{\Sigma}$ be iISS. Then $\Sigma$ is also iISS with the same $\beta, \theta$ and 
$\mu$ in the estimate \eqref{iISS_Estimate}.
\end{proposition}
For the system $\hat{\Sigma}$ whose state space $\hat{X}$ is a linear
normed space, we define the notion of iISS in the same way as it is done in Definition 4 for systems on a Banach space $X$.

The proof of Proposition~\ref{DensityArg} is given in Appendix~\ref{app:Density Argument}.

Often only approximations of the state space are needed, therefore we state another useful proposition whose proof is analogous to the proof of Proposition~\ref{DensityArg}.
\begin{proposition}
\label{DensityArg_StateAppr}
Let $\Sigma$ depend continuously on initial states, $\hat{\Sigma}$ be iISS and $\hat{U}=U$. Then $\Sigma$ is also iISS with the same $\beta, \theta$ and $\mu$ in the estimate \eqref{iISS_Estimate}.
\end{proposition}


\section{ {\rm i}ISS and ISS of bilinear systems}
\label{SectLinSys}

\subsection{Linear systems}

We will begin with a class of linear systems. 
Consider system \eqref{InfiniteDim} with $f(x(t),u(t))=Bu(t)$:
\begin{equation}
\label{LinSys}
\begin{array}{l}
{\dot{x}(t)=Ax(t)+Bu(t),} \\
x(0)=\phi_0,
\end{array}
\end{equation}
where $x:\R_+ \to X$, $u:\R_+ \to U$, and $B:U \to X$ is a linear operator. 

\begin{proposition}
\label{ISS_Lp_spaces}
Let $B \in L(U,X)$. Then,  
\eqref{LinSys} is 0-UGASs $\Iff$ \eqref{LinSys} is ISS $\Iff$ \eqref{LinSys} is ISS w.r.t. $L_p(\R_+,U)$ for some $p \geq 1$. 
\end{proposition}

\begin{proof}
The fact that \eqref{LinSys} is 0-UGASs $\Iff$ \eqref{LinSys} is ISS has been proved in \cite[Proposition 3]{DaM13}.
Next, by definition, it is obvious that 
if \eqref{LinSys} is ISS w.r.t. $L_p(\R_+,U)$ for a real number $p \geq 1$, 
then it is 0-UGASs. 
To prove the converse, let \eqref{LinSys} be 0-UGASs. Then $T$ is an exponentially stable semigroup, see e.g. \cite[Lemma 1]{DaM13}. Thus
$\exists M,\lambda>0: \|T(t)\| \leq Me^{-\lambda t}$.

We estimate the solution $x(t)=\phi(t,\phi_0,u)$ of \eqref{LinSys}:
\begin{align*}
\|x(t)\|_X  = & \|T(t)\phi_0 + \int_0^t{T(t-r)B u(r)dr}\|_X \\
						\leq & \|T(t)\|\|\phi_0\|_X + \int_0^t{\|T(t-r)\|\|B\| \|u(r)\|_U dr}, \\
						\leq & M e^{-\lambda t} \|\phi_0\|_X + M \|B\| \int_0^t{e^{-\lambda (t-r)} \|u(r)\|_U dr}.
\end{align*}
Now, estimating $e^{-\lambda (t-r)}  \leq 1$, $r \leq t$ we obtain that \eqref{LinSys} is ISS w.r.t. $L_1(\R_+,U)$. 

To prove the claim for $p>1$, pick any $q \geq 1$ so that $\frac{1}{p}+ \frac{1}{q} = 1$. 
We continue the above estimates, using the H\"older's inequality
\begin{align*}
\|x(t)\|_X  \leq & M e^{-\lambda t} \|\phi_0\|_X  + M \|B\| \int_0^t{e^{-\frac{\lambda}{2} (t-r)} e^{-\frac{\lambda}{2} (t-r)} \|u(r)\|_U dr}  \\
\leq & M e^{-\lambda t} \|\phi_0\|_X + M \|B\|  \Big(\int_0^t{e^{- \frac{q \lambda}{2} (t-r)}dr} \Big)^{\frac{1}{q}} \cdot \Big(\int_0^t{e^{-\frac{p \lambda}{2} (t-r)}\|u(r)\|^p_U dr} \Big)^{\frac{1}{p}} \\
 \leq & M e^{-\lambda t} \|\phi_0\|_X + M \|B\|  \Big( \frac{2}{q\lambda} \Big)^{\frac{1}{q}}  \Big(\int_0^t{\|u(r)\|^p_U dr} \Big)^{\frac{1}{p}} \\
 = & M e^{-\lambda t} \|\phi_0\|_X + w \|u\|_{L_p(\R_+,U)},
\end{align*}
where $w:=M \|B\|  \Big( \frac{2}{q\lambda} \Big)^{\frac{1}{q}}$.
This means that \eqref{LinSys} is ISS w.r.t. $L_p(\R_+,U)$ also for $p>1$. 
\end{proof}

The above result resembles the corresponding result for finite-dimensional systems. But in contrast to finite-dimensional case, 0-GAS systems \eqref{LinSys} with inputs of arbitrarily small magnitude may produce unbounded trajectories, even for bounded operators $B$
\cite[p. 8]{DaM13} and \cite[p. 247]{MaP11}.


The following statement is a consequence of 
Proposition~\ref{ISS_Lp_spaces} 
(by taking $\theta:= {\rm id}$ and $\mu := c\cdot {\rm id}$ for large enough $c>0$ in Definition \ref{Def:iISS}):
\begin{corollary}
System \eqref{LinSys} is ISS iff it is iISS.
\end{corollary}

For finite-dimensional systems, in the presence of nonlinearities which are locally Lipschitz w.r.t. state, 
0-GAS implies local ISS \cite[Lemma I.1]{SoW96}, 
i.e., the ISS property for initial states and inputs with a sufficiently small norm. 
In contrast to this finite-dimensional fact, we next show 
an infinite-dimensional linear system illustrating that 
\textit{for unbounded operator $B$, 0-UGASs implies neither ISS nor iISS, even if 
the initial state and the input is restricted to 
sufficiently small neighborhoods of the origin. 
}

\begin{examp}
\label{examp:Linear_Unbounded_InpOpe_nonISS}
Consider the following ODE ensemble defined on 
the interval $(0,\pi/2)$ of the spatial variable $l$: 
\begin{eqnarray}
\dot{x}(l,t)=-x(l,t)+(\tan l)^{\frac{1}{8}} u(l,t), \quad l \in (0,\pi/2) . 
\label{eq:Example_Unbounded_Input_operator}
\end{eqnarray}
Let $X=C(0,\pi/2)$ be the space of bounded continuous 
functions on $(0,\pi/2)$. The functions $x(l,t)$ and $u(l,t)$ are scalar-valued. 
The input operator $B:D(B) \to X$ for \eqref{eq:Example_Unbounded_Input_operator}
is defined by $(Bv)(l)=(\tan l)^{\frac{1}{8}} v(l)$ which 
is unbounded with a domain of definition
\[
D(B)=\{ v \in C(0,\pi/2): \sup_{l \in (0,\pi/2)}|(\tan l)^{\frac{1}{8}} v(l)| < \infty \}.
\]
Since $x(\cdot,t)=e^{-t}x(\cdot,0)$ holds for 
$u(\cdot,t)=0$, $\forall t\geq 0$,
system \eqref{eq:Example_Unbounded_Input_operator} is 
0-UGASs. But it is neither ISS nor iISS for $U=D(B)$. To verify this fact, consider an input $u(l,t)=\hat{u}_c(l)$ given by 
\[
\hat{u}_c(l)=\left\{ \begin{array}{ll}
b &,  \ 0 < l < \arctan{(c^8)}  \\ 
bc(\tan l)^{-\tfrac{1}{8}}   &, \ \arctan{(c^8)} \leq l < \frac{\pi}{2}
\end{array}\right.
\]
for real $b, c>0$ (we do not reflect the dependence of $\hat{u}_c$ on $b$ in the notation for the sake of simplicity). 
It is easy to see that $\hat{u}_c\in D(B)$ and $\|\hat{u}_c\|_U=b$ from 
$\|B\hat{u}_c\|_X=\sup_{l \in (0,\pi/2)}|\hat{u}_c(l)(\tan l)^{\frac{1}{8}}|=bc$ 
and the definition of $\hat{u}_c$. 
The solution of \eqref{eq:Example_Unbounded_Input_operator} for 
$\phi_0=0$ is computed as 
$\phi(t,0,u)(l)=\int_0^t e^{-(t-r)}\hat{u}_c(l)(\tan l)^{\frac{1}{8}}  dr=(1-e^{-t})\hat{u}_c(l)(\tan l)^{\frac{1}{8}}$. Thus, by definition, the solution satisfies 
\[
\sup_{l \in (0,\pi/2)}\phi(t,0,u)=bc(1-e^{-t}). 
\]
Now, assume that system \eqref{eq:Example_Unbounded_Input_operator} is iISS.  
From Definition \ref{Def:iISS} it follows that there exist 
$\theta$, $\mu \in \Kinf$ satisfying 
$\|\phi(t,0,u)\|_X \leq \theta (t \mu(b))$ for $t \geq 0$.
Clearly, for any given $\alpha$, $\mu \in \Kinf$ and any $t>0$, 
one can find $c=c(b,t)>0$ so that $bc(1-e^{-t}) > \theta (t \mu(b))$. 
Since $\|\hat{u}_c\|_U=b$ is satisfied for any $c>0$, 
system \eqref{eq:Example_Unbounded_Input_operator} is not iISS. 

Next, suppose that system \eqref{eq:Example_Unbounded_Input_operator} is ISS.  
Definition \ref{Def:ISS} implies the existence of 
$\gamma \in \Kinf$ satisfying 
$\|\phi(t,0,u)\|_X \leq \gamma(b)$ for $t \geq 0$ 
when $u=\hat{u}_c$ is applied to \eqref{eq:Example_Unbounded_Input_operator}. 
For any given $\gamma \in \Kinf$ and any $t>0$, 
there exists $c>0$ such that $bc(1-e^{-t}) > \gamma(b)$. 
Hence, system \eqref{eq:Example_Unbounded_Input_operator} is not ISS either. 
Since we can take $\phi_0=0$ and $b>0$ is arbitrary, the system 
\eqref{eq:Example_Unbounded_Input_operator} is 
neither iISS nor ISS even if the initial state and the input 
are restricted to arbitrarily small neighborhoods of the origin. 
Next we show that \eqref{eq:Example_Unbounded_Input_operator} is ISS if we choose $X=L_2(0,\pi/2)$, $U=L_4(0,\pi/2)$. 

%

Choose 
\[
V(x)=\int_0^{\pi/2}{x^2(l)dl} = \|x\|^2_{L_2(0,\pi/2)}
\]
For the solutions $x(\cdot,t)=\phi(t,\phi_0,u)$ of 
\eqref{eq:Example_Unbounded_Input_operator} we obtain 
\begin{align*}
\frac{d}{dt}V(x) &= 2 \int_0^{\pi/2}{x(l,t) \Big(-x(l,t) + u(l,t)(\tan l)^{\frac{1}{8}} \Big)dl} \\
&\leq -2 V(x) + w V(x) + \frac{1}{w}\int_0^{\pi/2}{u(l,t)^2(\tan l)^{\frac{1}{4}} dl}\\
&\leq (-2 +w) V(x) + \frac{K}{w} \|u(\cdot,t)\|^2_{L_4(0,d)},
\end{align*}
for any $w>0$ (between lines 1 and 2 Young's inequality has been used). Here, it is important to notice that 
$K:=\int_0^{\pi/2}{(\tan l)^{\frac{1}{2}}dl}< \infty$. Hence, 
taking $w<2$, Proposition \ref{PropSufiISS} proves that 
system \eqref{eq:Example_Unbounded_Input_operator} is ISS 
for $X=L_2(0,\pi/2)$ and $U=L_4(0,\pi/2)$. 
\qed
\end{examp}

\subsection{iISS of generalized bilinear systems}\label{sec:Binlin}

While for linear infinite-dimensional systems with bounded input operators the properties of ISS and iISS coincide, 
the difference between these two properties arises for 
bilinear systems which is one of 
the simplest classes of nonlinear systems. 
For finite-dimensional bilinear systems, 
Sontag \cite{Son98} demonstrated that 0-GAS systems  are seldom ISS (for example, $\dot{x}=-x+xu$, $x(t)\in \R$ is not ISS), but are always iISS. 
To generalize this result to infinite-dimensional systems, 
consider 
\begin{equation}
\label{BiLinSys}
\begin{array}{l}
{\dot{x}(t)=Ax(t)+ Bu(t) + C(x(t),u(t)),} \\
x(0)=\phi_0,
\end{array}
\end{equation}
where $B \in L(U,X)$, and $C: X \times U \to X$ is such that there exist $K>0$ and $\xi \in \K$ so that 
for all $x \in X$ and all $u \in U$ we have that
\begin{align}
\|C(x,u)\|_X \leq K \|x\|_X \xi(\|u\|_U).
\label{eq:BilinOperator}
\end{align}

\begin{remark}
This class includes systems with $C$ linear in both variables and bounded in the sense that $\|C\|:=\sup_{\|x\|_X =1,\|u\|_U=1}\|C(x,u)\|_X < \infty$  (then $K=\|C\|$ and $\xi(r)=r$ for all $r \in \R_+$). 
\end{remark}


%

Next we prove the equivalence between 0-UGASs and iISS for the general system \eqref{BiLinSys} in Banach spaces. 
For infinite-dimensions, we employ the notion of 0-UGASs instead of 0-GAS. 
To establish iISS from 0-GAS for finite-dimensional systems, 
Sontag \cite{Son98} constructed Lyapunov functions for systems 
with Hurwitz $A$ by means of the Lyapunov equation. 
To the best of authors' knowledge, there is no generalization of the method for construction of Lyapunov functions by means of the Lyapunov equation for linear systems on Banach spaces, although for systems on Hilbert spaces such a construction exists.
Thus, to allow for Banach spaces, we employ another method to prove that iISS is equivalent to 0-UGASs. 

\begin{Satz}
System \eqref{BiLinSys} is iISS $\Iff$ \eqref{BiLinSys} is 0-UGASs.
\end{Satz}

\begin{proof}
Clearly, iISS implies 0-UGASs for \eqref{BiLinSys}. To prove the converse, assume that 
\eqref{BiLinSys} be 0-UGASs, that is 
let $T$ be an exponentially stable semigroup generated by $A$.
Integrating \eqref{BiLinSys}, we obtain
\begin{equation*}
\label{iISS_Syst_IntEq}
x(t)=T(t)x(0) + \int_0^t{T(t-r) \big( Bu(r)+ C(x(r),u(r))\big)dr} .
\end{equation*}
From $B \in L(U,X)$, inequality \eqref{eq:BilinOperator} 
and since $T$ is an exponentially stable semigroup it follows for some $K,M,\lambda>0$, that
\begin{align*}
\label{iISS_Syst_IntEq_Estim1}
\|x(t)\|_X & \leq  \|T(t)\|\|x(0)\|_X + \int_0^t \|T(t-r)\| \big( \|B\|\|u(r)\|_U+ \|C(x(r),u(r))\|_X\big)dr \\
& \leq  Me^{-\lambda t} \|x(0)\|_X + \int_0^t Me^{-\lambda (t-r)} \big( \|B\|\|u(r)\|_U + K \|x(r)\|_X \xi(\|u(r)\|_U)\big)dr . 
\end{align*}
We multiply both sides of the inequality by $e^{\lambda t}$ and define $z(t) = x(t) e^{\lambda t}$. From $\lambda>0$ we obtain
\begin{align*}
\|z(t)\|_X \leq & M \Big( \|z(0)\|_X + \|B\| \int_0^t e^{\lambda r} \|u(r)\|_U dr\Big)  + \int_0^t{MK \|z(r)\|_X \xi(\|u(r)\|_U)dr}.
\end{align*}
Since $q: t \mapsto M\big( \|z(0)\|_X + \|B\| \int_0^t e^{\lambda r} \|u(r)\|_U dr \big)$ is a non-decreasing function, Gronwall's inequality (see e.g. \cite[Lemma 2.7, p.42]{Tes12}) yields 
\begin{align*}
\|z(t)\|_X \leq& M\Big( \|z(0)\|_X + \|B\| \int_0^t e^{\lambda r} \|u(r)\|_U dr \Big)  e^{\int_0^t{MK \xi(\|u(r)\|_U)dr}}.
\end{align*}
Coming back to original variables and using $\lambda>0$, we have 
\begin{align*}
\|x(t)\|_X 
\leq & M\Big( e^{-\lambda t}\|x(0)\|_X + \|B\| \int_0^t e^{-\lambda (t-r)} \|u(r)\|_U dr\Big)  e^{\int_0^t{MK \xi(\|u(r)\|_U)dr}} \\
\leq & M\Big( e^{-\lambda t}\|x(0)\|_X + \|B\| \int_0^t \|u(r)\|_U dr\Big)  e^{\int_0^t{MK \xi(\|u(r)\|_U)dr}}.
\end{align*}
Applying to the both sides the function $\alpha \in \Kinf$ defined for all $r \geq 0$ as $\alpha(r)=\ln(1+r)$ results in 
\begin{align*}
\alpha(\|x(t)\|_X ) \leq& \ln\Big(1+ M\big( e^{-\lambda t}\|x(0)\|_X +\|B\| \int_0^t \|u(r)\|_U dr\big) e^{\int_0^t{MK \xi(\|u(r)\|_U)dr}}\Big).
\end{align*}
Now since for all $a,b \in \R_+$ it holds that
\[
\ln(1+ae^b) \leq \ln((1+a)e^b)=\ln(1+a) + b , 
\]
we obtain
\begin{align*}
\alpha(\|x(t)\|_X ) \leq& \ln\Big( 1+ M\big( e^{-\lambda t}\|x(0)\|_X  +\|B\| \int_0^t \|u(r)\|_U dr\big) \Big) \\
 & \quad\quad + \int_0^t{MK \xi(\|u(r)\|_U)dr}.
\end{align*}
Moreover, for all $a,b \in \R_+$ it holds that
\[
\ln(1+a + b) \leq \ln((1+a)(1+b))=\ln(1+a) + \ln(1+b),
\]
which implies
\begin{align*}
\alpha(\|x(t)\|_X ) \leq& \ln\big(1+ M e^{-\lambda t}\|x(0)\|_X\big)\\
& + \ln\Big(1+ M\|B\| \int_0^t \|u(r)\|_U dr\Big)  + \int_0^t{MK \xi(\|u(r)\|_U)dr}.
\end{align*}
Since $\beta:(s,t) \mapsto \ln(1+ M e^{-\lambda t}s)$ is a $\KL$-function, with the help of 
$\alpha^{-1}(a+b)\leq \alpha^{-1}(2a)+\alpha^{-1}(2b)$ holding for 
any $a, b\in\R_+$, 
the above estimate shows us that \eqref{BiLinSys} is iISS 
as defined in Definition \ref{Def:iISS}.  
\end{proof}

\subsection{Lyapunov functions for generalized bilinear systems}\label{sec:Lyap_Binlin}

This section develops a method to construct an 
iISS Lyapunov function for the infinite-dimensional system 
\eqref{BiLinSys} 
analogous to the finite-dimensional case \cite{Son98}. 
For this purpose, in this section, let $X$ be a Hilbert space with a 
scalar product $\lel \cdot, \cdot \rir$. Note that 
if \eqref{BiLinSys} is 0-UGASs, the operator $A$ generates an exponentially stable semigroup \cite[Lemma 1]{DaM13}. 
Since $X$ is a Hilbert space, the exponential stability of this semigroup is equivalent to existence of a positive self-adjoint operator $P\in L(X)$ satisfying the Lyapunov equation
\begin{equation}
\label{OperatorBedingung}
\lel Ax,Px \rir + \lel Px,Ax \rir=-\|x\|^2_X \quad \forall x \in D(A),
\end{equation}
see \cite[Theorem 5.1.3, p. 217]{CuZ95}.
Recall that a self-adjoint operator $P \in L(X)$ is said to be positive if $\lel P x,x \rir > 0$ holds for all $x \in X\setminus\{0\}$. A positive operator $P \in L(X)$ is called coercive if there exists $k>0$ such that
\[
\lel Px,x \rir \geq k \|x\|^2_X \quad \forall x \in D(P).
\]

\begin{Satz}
\label{Thm:LF_Constructions_for_bilinear_systems}
Consider a system \eqref{BiLinSys} over a Hilbert space $X$. 
Let Assumption~\ref{Assumption1} hold and let 
there exists a coercive positive self-adjoint operator $P \in L(X)$ satisfying \eqref{OperatorBedingung}. 
If $A$ generates an analytic semigroup, then the system \eqref{BiLinSys} is iISS and its iISS Lyapunov function can be constructed as
\begin{align}
W(x)=\ln\Big(1+ \lel Px,x \rir \Big).
\label{eq:iISS_LF}
\end{align}
If $A$ does not necessarily generate an analytic semigroup, but all the trajectories, emanating from $D(A)$ under arbitrary input $u \in U_c$ remain in $D(A)$, then \eqref{BiLinSys} is still iISS and $W$ is its iISS Lyapunov function on $D(A)$.
\end{Satz}

\begin{proof}
Let assumptions of the theorem hold. Consider a function $V: x \mapsto \lel Px,x \rir $. 
Since $P$ is bounded and coercive, for some $k>0$ it holds 
\[
k \|x\|^2_X \leq V(x) \leq \|P\| \|x\|^2_X, \quad \forall x \in X,
\]
and property \eqref{LyapFunk_1Eig} is verified.
Let us compute the Lie derivative of $V$ with respect to the system \eqref{BiLinSys}. 
For $x \in D(A)$ we have
\begin{align*}
\dot{V}(x) =& \lel P\dot{x},x \rir +  \lel Px,\dot{x} \rir \\
 =&\lel P(Ax +Bu + C(x,u)),x \rir  + \lel Px,Ax +Bu + C(x,u) \rir  \\
=& \lel P(Ax),x \rir +  \lel Px,Ax \rir + \lel P(Bu + C(x,u)),x \rir +  \lel Px,Bu + C(x,u) \rir.
\end{align*}
From $\lel P(Ax),x \rir = \lel Ax,Px \rir$, 
\eqref{eq:BilinOperator} and \eqref{OperatorBedingung} 
with the help of Cauchy-Schwarz inequality, we obtain 
\begin{align*}
\dot{V}(x) \leq& -\|x\|^2_X + \|P(Bu + C(x,u))\|_X  \|x\|_X + \|Px\|_X \|Bu + C(x,u)\|_X \\ 
\leq& -\|x\|^2_X + \|P\| \|Bu + C(x,u)\|_X  \|x\|_X  + \|P\| \|x\|_X \|Bu + C(x,u)\|_X \\ 
\leq& -\|x\|^2_X + 2 \|P\| \|x\|_X (\|B\| \|u\|_U + K \|x\|_X \xi(\|u\|_U)) \\
\leq& -\|x\|^2_X \!\!+\! 2 K \|P\| \|x\|^2_X \xi(\|u\|_U)  \!\!+\! 2 \|P\| \|B\| \|x\|_X \|u\|_U.
\end{align*}
Let $\eps>0$. 
Using Young's inequality 
\begin{align*}
2 \|x\|_X \|u\|_U \leq \eps \|x\|^2_X + \dfrac{1}{\eps} \|u\|^2_U , 
\end{align*}
we can continue the above estimates as 
\begin{align*}
\dot{V}(x) \leq 
& -\Big(1 - \eps \|P\|\|B\|\Big) \|x\|^2_X + 2 K \|P\| \|x\|^2_X \xi(\|u\|_U) + \dfrac{\|P\|\|B\|}{\eps} \|u\|^2_U . 
\end{align*}
Defining $W$ as in \eqref{eq:iISS_LF} yields 
\begin{align*}
\dot{W}(x) \leq& \frac{1}{1+V(x)}
\Bigl[-\left(1 - \eps \|P\|\|B\|\right) \|x\|^2_X  + 2 K \|P\| \|x\|^2_X \xi(\|u\|_U) 
+ \dfrac{\|P\|\|B\|}{\eps} \|u\|^2_U \Bigr] \\
\leq & -\left(1 - \eps \|P\|\|B\|\right)\frac{\|x\|^2_X}{1+V(x)} + \frac{2 K \|P\|\|x\|^2_X}{1+k\|x\|^2_X}\xi(\|u\|_U) 
+ \dfrac{\|P\|\|B\|}{\eps} \|u\|^2_U. 
\end{align*}
which finally leads to
\begin{align}
\dot{W}(x) \leq& 
-\left(1 - \eps \|P\|\|B\|\right)\frac{\|x\|_X^2}{1+\|P\|\|x\|_X^2} + \frac{2 K \|P\|}{k}\xi(\|u\|_U) 
+ \dfrac{\|P\|\|B\|}{\eps} \|u\|^2_U. 
\label{LyapEstimate}
\end{align}
These derivations hold for $x \in D(A) \subset X$. 
If $x \notin D(A)$, then for all admissible $u$ the solution $x(t) \in D(A)$ and $t \to W(x(t))$ is a continuously differentiable function for all $t >0$ (these properties follow from the properties of solutions $x(t)$, see Theorem 3.3.3 in \cite{Hen81}). 
Therefore, by the mean-value theorem, $\forall t>0$ $\exists t_* \in (0,t)$
\[
\frac{1}{t}(W(x(t))-W(x))=\dot{W}(x(t_*)),
\]
where $x=x(0)$. 
Taking the limit when $t \to +0$ we obtain that \eqref{LyapEstimate} holds for all $x \in X$.
Pick $\eps>0$ such that $\eps<{1}/(\|P\|\|B\|)$. 
According to Proposition \ref{PropSufiISS}, system \eqref{BiLinSys} 
is iISS and $W$ is an iISS Lyapunov function. 

Let now all the trajectories, emanating from $D(A)$ under arbitrary input $u \in U_c$ remain in $D(A)$. This means, that \eqref{BiLinSys} generates a control system on $D(A)$ with inputs $U_c$. 

According to above argument $W$ is an iISS Lyapunov function for this restricted system, which shows its iISS. 
Similarly to \cite[Proposition 4.3.7]{CaH98} one can prove, that Assumption~\ref{Assumption1} implies that \eqref{BiLinSys} depends continuously on initial data. Thus, due to Proposition~\ref{DensityArg_StateAppr} iISS of \eqref{BiLinSys} for state from $D(A)$ and inputs belonging to $U$ implies iISS of \eqref{BiLinSys} on the whole spaces $X,U$. 
\end{proof}

As we see the proof of the previous theorem consists basically of two parts: the verification that $W$ is an iISS Lyapunov function for $x \in D(A)$ and then use of a density argument for analytic semigroups.
We will see that this strategy will be useful also in the proof of Theorem~\ref{thm:Dissipative_Implicative}, see Appendix.

\begin{remark}
Note, that in the case when $A$ generates a nonanalytic semigroup we do not claim that $W$ is an iISS Lyapunov function on the whole space $X$, since we cannot compute the Lie derivative $\dot{W}_u$ for $x \notin D(A)$.
\end{remark}
%
%
%

\section{An example}
\label{sec:Example}

In this concluding section we illustrate our findings on an example of an iISS parabolic system. 
Let $c>0$ and $L>0$. 
Consider the following reaction-diffusion system
\begin{equation}
\label{GekoppelteNonLinSyst}
\left
\{
\begin{array}{l} 
{\dfrac{\partial x}{\partial t}(l,t) = c \dfrac{\partial^2 x}{\partial l^2}(l,t) + \dfrac{x(l,t)}{1+|l-1|x(l,t)^2}u(l,t),} \\[1.5ex]
{x(0,t) = x(L,t)=0;} 
\end{array}
\right.	
\end{equation}
on the region $(l,t) \in (0,L) \times (0,\infty)$ of 
the $\R$-valued functions $x(l,t)$ and $u(l,t)$. 

Let $X=L_2(0,L)$ and $U=C(0,L)$.
It is easy to see that the above system is a generalized bilinear system since its nonlinearity satisfies inequality \eqref{eq:BilinOperator}. Clearly, this system is 0-UGASs, therefore it is iISS for any $L>0$. Below we give an explicit construction of an iISS Lyapunov function for this system. Afterwards we will prove that this system is ISS for $L<1$.

Define 
\[
W(x)=\int_0^L{x^2(l)dl} = \|x\|^2_{L_2(0,L)} . 
\]
Since $1+|l-1|x(l,t)^2\geq 1$, we obtain 
\begin{align*}
\dot{W}(x) =& 2 \int_0^L{x(l)\Big( c \dfrac{\partial^2 x}{\partial l^2}(l,t) + \dfrac{x(l,t)}{1+|l-1|x(l,t)^2}u(l,t)\Big)}dl \\
\leq& -2c  \int_0^L \left( \dfrac{\partial x}{\partial l}(l,t) \right)^2 dl + 2 \int_0^L{x^2(l,t)|u(l,t)|dl}.
\end{align*}
Applying the Friedrich's inequality \cite[p. 67]{MPF91} to the first term, we proceed to:
\begin{align*}
\dot{W}(x) \leq  -2c \left( \frac{\pi}{L} \right)^2 W(x) + 2 W(x) \|u\|_{C(0,L)}
\end{align*}
Choosing  
\begin{align}
V(x)=\ln\left(1+W(x)\right)  
\label{eq:LF_1st_Sys}
\end{align}
yields 
\begin{align}
\dot{V}(x) \leq& 
-2c \!\left( \frac{\pi}{L} \right)^2\!\frac{ W(x)}{1\!+\!W(x)}
\!+2 \frac{ W(x)}{1\!+\!W(x)} \|u\|_{C(0,L)}    \nonumber\\
\leq &  -2c \!\left( \frac{\pi}{L} \right)^2\!\frac{\|x\|_{L_2(0,L)}^2}{1\!+\!\|x\|_{L_2(0,L)}^2}
+   2 \|u\|_{C(0,L)},
\label{Examp_Conclusing_iISS_estim} \\
=&
-\alpha(\|x\|_{L_2(0,L)}) + \sigma(\|u\|_{C(0,L)}) , 
\nonumber 
\end{align}
where 
\begin{align}
\alpha(s)=
2c \!\left( \frac{\pi}{L} \right)^2\!\frac{s^2}{1\!+\!s^2}
, \quad 
\sigma(s)=2s . 
\label{Examp_Conclusing_iISS_estimb}
\end{align}
Thus, Proposition \ref{PropSufiISS}
establishes iISS of \eqref{GekoppelteNonLinSyst} irrespective of a value of $L$.

\begin{remark}
Since $x(\cdot,t) \in L_2(0,L)$, the spatial derivative of $x$ above may not exist. However, the above derivations hold for smooth enough functions $x$, and the general result for all $x(\cdot,t) \in L_2(0,L)$ will follow due to Proposition~\ref{DensityArg}.
\end{remark}

Interestingly, when $L<1$, the system \eqref{GekoppelteNonLinSyst} is ISS for 
the input space $U=C(0,L)$ as well as $U=L_2(0,L)$. 
To verify this, we first note that 
\begin{align}
\sup_{s\in\R}
\left|\dfrac{s}{1+|l-1|s^2}\right|=
\dfrac{1}{2\sqrt{1-l}}  
\end{align}
holds for $l<1$. 
Assume $L<1$.  Using the same Lyapunov function $W$ we obtain 
\begin{align}
\dot{W}(x) \leq & 
2 \int_0^L x(l,t)c \frac{\partial^2 x}{\partial l^2}(l,t)dl 
+ 2 \int_0^L \dfrac{1}{2\sqrt{1-l}}|x(l,t)u(l,t)|dl
\nonumber \\
\leq& 
-2c \left( \frac{\pi}{L} \right)^2 \|x\|_{L_2(0,L)}^2
+\dfrac{1}{\sqrt{1-L}}\|x\|_{L_2(0,L)}\|u\|_{L_2(0,L)}
\nonumber \\
\leq& 
-\left(2c \left( \frac{\pi}{L} \right)^2-w\right) \|x\|_{L_2(0,L)}^2
+\dfrac{1}{4(1-L)w}\|u\|_{L_2(0,L)}^2
\label{Example_Estim}
\end{align}
for $0<w<2c({\pi}/{L})^2$.  Recall that 
$\|u\|_{L_2(0,L)}^2\leq L\|u\|_{C(0,L)}^2$ for $u \in C(0,L)$. 
Thus, by virtue of Proposition \ref{PropSufiISS}, 
system \eqref{GekoppelteNonLinSyst} is ISS whenever $L<1$. It is 
stressed that the coefficient of $\|u\|_{L_2(0,L)}^2$ 
in \eqref{Example_Estim} goes to $\infty$ as $L$ tends to $1$ from below. Hence, the ISS estimate \eqref{Example_Estim} is valid only if $L<1$. 

For the choice of input space $U=L_p(0,L)$ with $p \geq 1$, the case of 
$L \geq 1$ does not allow us to have an ISS estimate like 
\eqref{Example_Estim}. 
In fact, if $L \geq 1$ and $U=L_p(0,L)$ for any $p \geq 1$, the right hand side of \eqref{GekoppelteNonLinSyst} system is undefined. 

To see this take $u:l \mapsto |l-1|^{-\tfrac{1}{2p}} \in L_p(0,L)$ and 
$x:l \mapsto |l-1|^{-\tfrac{1}{2}+\tfrac{1}{2p}} \in L_2(0,L)$.
Then $f(x,u):l \mapsto \dfrac{x(l,t)}{1+|l-1|x(l,t)^2}u(l,t) \notin L_2(0,L)$.
Thus, according to our formulation of \eqref{InfiniteDim}, the system \eqref{GekoppelteNonLinSyst} is not well-defined for $U=L_p(0,L)$  for any real $p \geq 1$.

For the choice of input space $U=C(0,L)$, we expect that the
system \eqref{GekoppelteNonLinSyst} is not ISS for $L \geq 1$, 
but we have not proved it at this time. 
The blow-up of the $\sigma$-term in \eqref{Example_Estim} corresponding to 
the dissipation inequality \eqref{DissipationIneq} for $V=W$ suggests the 
absence of ISS for the system \eqref{GekoppelteNonLinSyst} in the case of 
$L \geq 1$. 
It is worth noticing that the iISS estimate 
\eqref{Examp_Conclusing_iISS_estim} is valid for all $L>0$, 
that is for all $L>0$ we do no have in \eqref{Examp_Conclusing_iISS_estimb} a blowup of $\sigma$, and $\alpha$ does not
become a zero function. In fact, one can recall the idea 
demonstrated by Proposition~\ref{PropSufiISS} with 
Definition~\ref{def:iISSV}. 
An iISS Lyapunov function characterizes the absence of ISS by 
only allowing the decay rate $\alpha$ in the dissipation 
inequality \eqref{DissipationIneq} to satisfy 
$\liminf_{s\to\infty}\alpha(s)<\lim_{s\to\infty}\sigma(s)$. 
The iISS Lyapunov function yields ISS when 
$\liminf_{s\to\infty}\alpha(s)\geq \lim_{s\to\infty}\sigma(s)$, which is not the case in \eqref{Examp_Conclusing_iISS_estimb}. 

Being able to uniformly characterize iISS irrespectively of whether 
systems are ISS or not 
should be advantageous in many applications. For instance, 
ISS of subsystems is not necessary for stability of 
their interconnections, and there are examples of 
UGAS interconnections involving iISS systems which are not ISS 
\cite{Ito06,ItJ09,AnA07,KaJ12}.

\section{Conclusion}
\label{sec:Conclusion}

We have proved that infinite-dimensional bilinear systems described by differential equations in Banach spaces are integral input-to-state stable provided they are uniformly globally asymptotically stable. For the systems whose state space is Hilbert we have obtained under some additional restrictions another proof of this result, which leads to a construction of an iISS Lyapunov function for the system. 

The possible directions for future research are investigation of iISS of more general nonlinear control systems and development of novel methods for construction of iISS Lyapunov functions for such systems. Another challenging problem is a study of interconnected infinite-dimensional systems, whose subsystems are iISS or ISS. These questions have been treated in a recent paper \cite{MiI14d} for interconnections of two parabolic systems with 1-dimensional spatial domain. However, many interesting questions are still open.

\section{Appendix}

\subsection{Proof of Proposition~\ref{DensityArg}}
\label{app:Density Argument}

%
%


Since $\hat{\Sigma}$ is iISS, there exist $\beta \in \KL$ and $\mu\in\K$, $\theta \in \Kinf$, such that $\forall \hat{x} \in \hat{X},\; \forall \hat{u} \in \hat{U}_c$ and $\forall t\geq 0$ it holds that 
\begin {equation}
\label{DichteArg_1}
\| \phi(t,\hat{x},\hat{u}) \|_{X} \leq \beta(\| \hat{x} \|_{X},t) + \theta \Big(\int_0^t \mu(\|\hat{u}(s)\|_{U})ds\Big).
\end{equation}
Let $\Sigma$ be not iISS with the same $\beta,\mu,\theta$. Then there exist $t^*>0$, $x\in X$, $u \in U_c$:
\begin {equation}
\label{DichteArg_2}
\| \phi(t^*,x,u) \|_{X} = \beta(\| x \|_{X},t^*) + \theta \Big(\int_0^{t^*} \mu(\|u(s)\|_U)ds\Big) + r,
\end{equation}
where $r=r(t^*,x,u)>0$.
From \eqref{DichteArg_1} and \eqref{DichteArg_2} we obtain
\begin{align}
\| \phi(t^*,x,u) \|_{X} - & \| \phi(t^*,\hat{x},\hat{u}) \|_{X} \geq 
\beta(\| x \|_{X},t^*) - \beta(\| \hat{x} \|_{X},t^*) 
\nonumber \\
&\quad + \theta \Big(\int_0^{t^*}\hspace{-1.2ex}\mu(\|u(s)\|_U)ds\Big) - \theta \Big(\int_0^{t^*}\hspace{-1.2ex}\mu(\|\hat{u}(s)\|_{U})ds\Big) + r.
\label{DichteArg_3}
\end{align}
Since $\hat{X}$ and $\hat{U}_c$ are dense in $X$ and $U_c$ respectively, 
and since operator $u \mapsto \theta \Big(\int_0^{t^*} \mu(\|u(s)\|_U)ds\Big)$ is continuous, we can find sequences $\{\hat{x}_i \} \subset \hat{X}$: $\|x-\hat{x}_i \|_X \to 0$ and $\{\hat{u}_i \} \subset \hat{U}_c$: $\|u-\hat{u}_i \|_{U_c} \to 0$.
From \eqref{DichteArg_3} it follows that for each arbitrary $\eps>0$, 
there exist $\hat{x}_i$ and $\hat{u}_i$ such that 
\begin{equation*}
\| \phi(t^*,x,u) - \phi(t^*,\hat{x}_i,\hat{u}_i) \|_{X}  \geq
\left| \| \phi(t^*,x,u) \|_{X} - \| \phi(t^*,\hat{x}_i,\hat{u}_i) \|_{X}  \right| \\
\geq r - 2\eps.
\end{equation*}
This contradicts to the assumption of continuous dependence of $\Sigma$ on initial states and inputs.
Thus, $\Sigma$ is iISS with the same $\theta$, $\beta$ and $\mu$ 
in \eqref{iISS_Estimate}.

\subsection{Proof of Theorem~\ref{thm:Dissipative_Implicative}}
\label{app:Implication_Dissipation}

%
%
%

First, assume that $V$ is an ISS Lyapunov function in the 
implicative formulation for \eqref{InfiniteDim}. 
Pick any $x \in X$ and $u \in U_c$ s.t. $\|x\|_X \leq \gamma(\|u(0)\|_U)$.
Since $V$ is Fr\'echet differentiable in $X$ and since for $x \in D(A)$ the trajectory $\phi(\cdot,x,u)$ is differentiable, 
 $V(x(t))$ is also differentiable (see \cite[par. 2.2]{ATF87}) and can be computed as
\begin{align}
\dot{V}_u(x) &=  \frac{\partial V}{\partial x}(x) \big(Ax + f(x,u(0))\big) 
\nonumber \\
&=  \frac{\partial V}{\partial x}(x) (Ax + f(x,0)) 
+ \frac{\partial V}{\partial x}(x)( f(x,u(0)) - f(x,0)) 
\nonumber \\
&\leq  -\eta(\|x\|_X) + \Big\| \frac{\partial V}{\partial x}(x)\left(f(x,u(0)) - f(x,0)\right)\Big\|_X
\label{TempEstimate-1}
\end{align}
Here $\frac{\partial V}{\partial x}(x)$ denotes a Fr\'echet derivative of $V$ at point $x \in X$ (which is a bounded linear operator from $X$ to $\R$ with an operator norm $\|\cdot\|$). 
Since $\frac{\partial V}{\partial x}$ is bounded on bounded balls, there exists $q_2 \in \K$ such that 
\begin{eqnarray*}
\Big\| \frac{\partial V}{\partial x}(x) \Big\| \leq q_2(\|x\|_X) .
\end{eqnarray*}
Moreover, since $f$ is Lipschitz w.r.t. $x$ on bounded subsets of $X$, we have
\begin{eqnarray*}
\|f(x,u(0)) - f(0,u(0))\|_X \leq w(\|x\|_X)\|x\|_X
\end{eqnarray*}
for some continuous non-decreasing function $w: \R_+\to\R_+$. 
Thus, we have 
\begin{eqnarray*}
\left\| \frac{\partial V}{\partial x}(x)\left(f(x,u(0)) - f(0,u(0))\right)\right\|_X \leq \hat{w}(\|x\|_X)\|x\|_X
\end{eqnarray*}
for some $\hat{w}\in\K$. 
This implies
\begin{align*}
 \left\| \frac{\partial V}{\partial x}(x)\left(f(x,u(0)) - f(x,0)\right)\right\|_X
\leq& \|\tfrac{\partial V}{\partial x}(x)(f(x,u(0)) - f(0,u(0)))\|_X \\
&\hspace{10ex}+ \|\tfrac{\partial V}{\partial x}(x)(f(0,u(0)) - f(x,0))\|_X  \\
\leq& \hat{w}(\|x\|_X)\|x\|_X + \|\tfrac{\partial V}{\partial x}(x)f(0,u(0))\|_X \\
&\hspace{12ex}+ \|\tfrac{\partial V}{\partial x}(x)f(x,0)\|_X \\
\leq& 2 \hat{w}(\|x\|_X)\|x\|_X + \|\tfrac{\partial V}{\partial x}(x)f(0,u(0))\|_X.  
\end{align*}
Due to \eqref{eq:Boundedness_wrt_u} we proceed from \eqref{TempEstimate-1} to
\begin{align*}
\dot{V}_u(x) \leq &-\eta(\|x\|_X) + 2 \hat{w}(\|x\|_X)\|x\|_X + p(\|x\|_X)+q(\|u(0)\|_U).
\end{align*} 
And for $\|x\|_X \leq \gamma(\|u(0)\|_U)$
we obtain finally 
\begin{eqnarray}
\label{AG-like_estimate}
\dot{V}_u(x) \leq -\eta(\|x\|_X)+\sigma(\| u(0)\|_U)
\end{eqnarray} 
with  $\sigma(r):= 2 \hat{w}(\gamma(r))\gamma(r) + p(\gamma(r))+q(r)$ 
which is of class $\K_\infty$. 
Pick $\alpha=\eta$, which is of class $K_\infty$ due to Definition 6. 
Combining \eqref{AG-like_estimate} with the implication \eqref{GainImplikation} we obtain that
for all $x \in D(A)$ and all $u \in U_c$
\begin{eqnarray*}
\dot{V}_u(x) \leq  -\alpha(\|x\|_X)  + \sigma(\| u(0)\|_U),
\end{eqnarray*}
i.e. $V$ is an ISS Lyapunov function in a dissipative form 
\eqref{DissipationIneq} for states from $D(A)$.

Since $A$ generated an analytic semigroup, we can apply a density argument for analytic semigroups as in the proof of Theorem~\ref{Thm:LF_Constructions_for_bilinear_systems} to prove that $V$ is an ISS Lyapunov function in a dissipative form for \eqref{InfiniteDim} on the whole $X$.

%
%
%
%

Now let us prove the converse. Basically we can follow 
the argument for finite-dimensional systems. 
Suppose that $V$ is an ISS Lyapunov function in a dissipative 
formulation, i.e., \eqref{DissipationIneq}. 
Due to $\sigma\in\K$, property \eqref{eq:ISSalphsig} ensures 
the existence of $\hat{\alpha}\in\K$ such that 
$\lim_{\tau\to\infty}\hat{\alpha}(\tau)\geq \lim_{\tau\to\infty}\sigma(\tau)$ and
\begin{align*}
\hat{\alpha}(s)\leq\alpha(s), \quad \forall s\in\R_+
\nonumber
\end{align*}
If either $\lim_{\tau\to\infty}\hat{\alpha}(\tau)=\infty$
or $\lim_{\tau\to\infty}\hat{\alpha}(\tau)>\lim_{\tau\to\infty}\sigma(\tau)$ holds, we can pick a constant $K>1$ such that 
$\lim_{\tau\to\infty}\hat{\alpha}(\tau)\geq K
\lim_{\tau\to\infty}\sigma(\tau)$.  
Then $V$ achieves \eqref{GainImplikation} 
with $\gamma=\hat{\alpha}^{-1}\circ K\sigma\in\K$ 
and $\eta=(1-1/K)\hat{\alpha}\in\K$.  
In the case of $\infty>\lim_{\tau\to\infty}\hat{\alpha}(\tau)=
\lim_{\tau\to\infty}\sigma(\tau)$, there exists a continuous function 
$\omega: \R_+\to\R_+$ satisfying ${\rm id}+\omega\in\K_\infty$ and  
\begin{align*}
& 
\omega(\alpha(s))>0, \quad s\in(0,\infty)  
\\
& 
\lim_{\tau\to\infty}\omega(\alpha(\tau))=0 .
\end{align*}
Then property \eqref{DissipationIneq} with 
$\gamma=\hat{\alpha}^{-1}\circ({\rm id}+\omega)\circ\sigma\in\K_\infty$ yields 
\eqref{GainImplikation} with 
$\eta=\left({\rm id}-({\rm id}+\omega)^{-1}\right)\circ\hat{\alpha}
\in\PD$. 
The function $\eta$ obtained in the above two cases is 
only guaranteed to satisfy either $\eta\in\K$ or $\eta\in\PD$. 
The function $V$ can be transformed into another continuous function 
$\hat{V}:X \to \R_+$ by applying the technique in \cite{SoT95} 
to obtain $\eta\in\K_\infty$ in \eqref{GainImplikation}. 
The transformed $\hat{V}$ is an ISS Lyapunov function of 
\eqref{InfiniteDim} in an implicative form.

\subsection{Proof of Lemma~\ref{thm:ComparisonPrinciple}}
\label{app:Comparison_Principle}

For a continuous function $y:\R \to \R$, let 
the right upper Dini derivative and 
the right lower Dini derivative be defined by 
$D^+y(t):=\limsup_{h \to +0}\frac{y(t+h)-y(t)}{h}$ and 
$D^-y(t):=\liminf_{h \to +0}\frac{y(t+h)-y(t)}{h}$, respectively. 
We first state the following simple lemma. 
\begin{Lemm}
\label{Lemma2}
Let $g$ be a continuous function on a bounded interval $[a,b]$ and let $t$ be such that $f: Im(g) \to\R$ is continuously differentiable at $g(t)$ and
$\frac{df}{dr}(r)\Big|_{r=g(t)} > 0$.
Then
\[
D^{+}f(g(t)) = \frac{df}{dr}(r)\Big|_{r=g(t)} D^{+}g(t).
\]
\end{Lemm}
\begin{proof}
Consider a function $h: t\in\R \to \R$ which is continuous at $t=c$ 
and satisfies $h(c) > 0$. Let $k:\R \to \R$ be an arbitrary function. 
Then the following holds (allowing the limits to be equal $\pm \infty$)
\[
\limsup_{t \to c}\big(h(t)k(t)\big) = \lim_{t\to c}h(t) \cdot \limsup_{t \to c}k(t).
\]
Thus, the claim follows from the definition of the right upper Dini derivative. 
\end{proof}

We also use the following proposition (see e.g. \cite[Corollary 4, p. 113]{RoF10}):  
\begin{proposition}
\label{prop:Relaxed_NL}
Let $f$ be an increasing continuous function on a bounded 
interval $[a,b]$. Then $f$ is differentiable almost everywhere 
on $(a,b)$. Furthermore, $\frac{df}{dt}$ is integrable over $[a,b]$ and
\begin{eqnarray}
\int_{t_0}^t \frac{df}{ds} ds 
\leq \int_{a}^b D^{-}f(s) ds 
\leq \int_{a}^b D^{+}f(s) ds 
\leq f(b) - f(a).
\label{eq:Relaxed_Newton-Leibniz}
\end{eqnarray}
\end{proposition}

Now we are in a position to prove Lemma~\ref{thm:ComparisonPrinciple}.
From \eqref{eq:ComparisonPrinciple} it follows that 
\begin{eqnarray}
\label{HelpIneq}
\frac{\dot{y}(t)}{\alpha(y(t))} \leq - 1 
\end{eqnarray}
for all $t$ so that $y(t) \neq 0$.
Define $\eta:\R_+ \to [-\infty,\infty)$ by 
\begin{eqnarray}
\eta(s):= \int_{1}^{s}\frac{dr}{\alpha(r)}.
\label{eq:eta_def}
\end{eqnarray}
Without loss of generality, we assume 
$\lim_{s\to +0}\eta(s)=-\infty$ and 
$\lim_{s\to +\infty}\eta(s)=+\infty$\footnote{
Property $\lim_{s\to +\infty}\eta(s)=+\infty$ is not necessary 
since \eqref{eq:ComparisonPrinciple-eq5} is defined for $t\in\R_+$.}. 
If this is not the case, 
following the idea in \cite[Lemma 4.4]{LSW96}, 
we can replace $\alpha$ with $\min\{s,\alpha(s)\}$ to obtain 
a positive definite function satisfying 
\eqref{eq:ComparisonPrinciple}, $\lim_{s\to +0}\eta(s)=-\infty$ 
and $\lim_{s\to +\infty}\eta(s)=+\infty$. 
Since $\alpha$ is continuous, $\eta$ is continuously differentiable and its derivative is positive on its domain of definition. Thus, Lemma~\ref{Lemma2} implies 
\[
\frac{\dot{y}(t)}{\alpha(y(t))} = D^+ \eta(y(t)) , 
\]
and \eqref{HelpIneq} can be rewritten as
\begin{eqnarray}
D^+ \eta(y(t)) \leq - 1 . 
\label{eq:eta_deni_negative}
\end{eqnarray}
On the other hand, from \eqref{eq:ComparisonPrinciple} we see that 
$y$ is a decreasing function as long as $y(t)>0$.
Since $\alpha \in \PD$, $\eta$ is strictly increasing. 
Hence, $\eta(y(\cdot))$ is a decreasing function. Applying 
Proposition~\ref{prop:Relaxed_NL} to $-\eta(y(\cdot))$ 
together with \eqref{eq:eta_deni_negative} yields 
\begin{eqnarray}
\eta(y(t)) - \eta(y(0)) 
\leq -\int_0^t D^- \{-\eta(y(s))\}ds 
\leq \int_0^t D^+ \eta(y(s))ds 
\leq -t.
\label{eq:ComparisonPrinciple-eq4}
\end{eqnarray}
Since $\eta$ is strictly increasing, $\eta$ is invertible and its inverse $\eta^{-1}$ is a strictly increasing function, defined over 
$[-\infty,+\infty)$. 
An easy manipulation with \eqref{eq:ComparisonPrinciple-eq4} leads to
\begin{eqnarray}
y(t) \leq \eta^{-1}( \eta(y(0))  - t) 
\label{eq:ComparisonPrinciple-eq5}
\end{eqnarray}
for all $t\in[0,T)$, where $T:=\min\{ t\in \R_+ : y(t)=0)\}$. 
Here, let $T=\infty$ if $\{ t\in \R_+ : y(t)=0)\}=\emptyset$. 
If $y(t)=0$ holds for some $T\in\R_+$, then 
$y(t)=0$ is satisfied for all $t\geq T$, due to 
\eqref{eq:ComparisonPrinciple} and $y(t)\geq 0$. 
Define $\beta:\R_+ \times \R_+ \to \R_+$ by
\[
\beta(r,t):=
\begin{cases}
\eta^{-1}( \eta(r)  - t) & \text{, if } r>0, \\ 
0                        & \text{, if } r=0.
\end{cases}
\]
Due to the strict increasing property of $\eta$ and 
$\lim_{s\to +0}\eta(s)=-\infty$, we have $\beta \in \KL$. 
With this $\beta$ the inequality 
\eqref{eq:ComparisonPrinciple_FinalEstimate} follows from 
\eqref{eq:ComparisonPrinciple-eq5} and $y(t)=0$, $\forall t\geq T$.

\bibliographystyle{abbrv}
\bibliography{Mir_LitList}

%

\end{document}